\documentclass{commat}

\usepackage[markup=underlined]{changes}
\usepackage[shortlabels]{enumitem}
\usepackage{todonotes}
\DeclareMathOperator{\disc}{disc}
\DeclareMathOperator{\Gal}{Gal}
\DeclareMathOperator{\Tr}{Tr}
\newcommand{\bigperp}{\mathop{\mathpalette\bigpirp\relax}\displaylimits}
\newcommand{\bigpirp}[2]{\vcenter{\hbox{\scalebox{\ifx#1\displaystyle2.1\else1.5\fi}{$#1\perp$}}}}
\newcommand{\cO}{\mathcal{O}}
\newcommand{\Q}{\mathbb{Q}}
\newcommand{\R}{\mathbb{R}}
\newcommand{\Z}{\mathbb{Z}}

\EditInfo{February 2, 2023}{May 18, 2023}{Camilla Hollanti and Lenny Fukshansky}

\VOLUME{31}
\NUMBER{2}
\firstpage{81}
\DOI{https://doi.org/10.46298/cm.10896}
\title{%
    Universal quadratic forms and indecomposables in number fields: A survey 
    }

\author{%
    V\' \i t\v ezslav Kala
    }

\authorinfo{%
    Charles University, Faculty of Mathematics and Physics, Department of Algebra}{%
    vitezslav.kala@matfyz.cuni.cz
    }

\abstract{%
    We give an overview of universal quadratic forms and lattices, focusing on the recent developments over the rings of integers in totally real number fields. In particular, we discuss indecomposable algebraic integers as one of the main tools.
    }



\begin{document}

The goal of this survey article is to give an overview of the arithmetic theory of universal quadratic forms. I will primarily focus on the results over number fields obtained since 2015.
	For other surveys focusing on different facets of the area, see~\cite{Ea, Han, Kim}.

	While I try to explain the broad ideas behind the proofs and the tools that are used, many details are nevertheless missing and there are numerous simplifications.
	The interested reader is therefore always encouraged to look into the original papers or to contact me.
	Overall the notes are more aimed at a~junior audience rather than the experts in the field (who might at least prefer to start reading only in Sections~\ref{sec:4} or~\ref{sec:5}).

	The paper is primarily based on the notes from my lectures at the \textit{XXIII International Workshop for Young Mathematicians} in Krakow, Poland (and from some of my other talks).
	Parts of the text are also taken from the introduction to my habilitation thesis~\cite{Ka3}.

	\section{Introduction}\label{sec:1}

	The study of representations of integers by quadratic forms has a~long history; let's briefly start here with a~few highlights.

	One can perhaps argue that they were first considered as Pythagorean triples, i.e., solutions of the Diophantine equation $x^2 +y^2 =z^2$, or, equivalently, representations of 0 by the indefinite ternary form $x^2 +y^2 -z^2$. A list of 15 such triples occurs already on the Babylonian clay tablet Plimpton~322~\cite{Rob} from around 1800~BC!

	The Pell equation, i.e., representation of $1$ and other small integers by the binary form $x^2 -dy^2$ (for some $d\in\Z_{>0}$ that is not a~square),
	was considered as early as 400~BC by Greek mathematicians in connection with approximating $\sqrt 2, \sqrt 3$ by rational numbers.
	Later it was studied, e.g., by Archimedes (3rd~century BC) and Diophantus (3rd~century AD) and in India by
	Brahmagupta (7th~century AD) and Bhaskara (12th~century AD).

\medskip

	The modern European history starts with giants such as Fermat, Euler, and Gauss, who considered representations of primes by binary definite forms $x^2 +dy^2$ (for $d\in\Z_{>0}$)~\cite{Cox} and obtained results such as \emph{a prime number $p$ is of the form $x^2 +y^2$ if and only if $p=2$ or $p\equiv 1\pmod 4$}.

	In 1770 Lagrange proved the Four--Square Theorem stating that \emph{every positive integer $n$ is of the form $x^2 +y^2 +z^2 +w^2$}; Jacobi then in 1834 gave a~formula for the number of representations of $n$
	in this form. In a~similar vein, Legendre in 1790 proved the Three--Square Theorem that characterizes the integers of the form $x^2 +y^2 +z^2$ (for much more information on the history, see, e.g.,~\cite{Di2}).
	These results eventually led, e.g., to the still active Waring problem, and to
	using modular forms for studying the representations of integers by quadratic forms.

 \medskip

	A quadratic form representing all positive integers is called \emph{universal}. Lagrange's Theorem thus says that $ x^2 +y^2 +z^2 +w^2 $ is a~universal quadratic form. {The early 20th century saw the characterization of all universal quaternary diagonal positive forms $ ax^2 +by^2 +cz^2 +dw^2 $ by Ramanujan~\cite{Ra}, and an extension of this work to non-diagonal forms by Dickson~\cite{Di1}, who also introduced the term ``universal quadratic form".} Among such forms are, for example, $ x^2 +2y^2 +4z^2 +dw^2 $ for $ 1 \leq d \leq 14 $. These results were further expanded by Willerding~\cite{Wi} to cover also the cases of classical quaternary forms (although her list contains a~number of errors, it was nevertheless a~big step towards the full classification).

	While no ternary positive definite quadratic form is universal (for local reasons), indefinite quadratic forms tend to be universal more easily. For example, any quadratic form $ x^2 -y^2 -dz^2 $ is universal as long as 4 does not divide $ d $ because $ x^2 -y^2 =(x+y)(x-y) $ represents all odd numbers. They form a~separate area of study, and we will return to them later only very briefly.

	\section{The 15- and 290-Theorems}\label{sec:2}

	Let's first discuss in some detail the case of quadratic forms over the ring of integers $\Z$.	Recall that
	a \emph{quadratic form} of rank $ r $ over $\Z$ is a~polynomial
\begin{equation}\label{eq:form}
		Q(x_1, \dots, x_r) = \sum_{1\leq i \leq j \leq r} a_{ij}x_ix_j,\qquad a_{ij}\in \Z.	
\end{equation}
	Typically we require the form to be \emph{positive definite}, meaning that $ Q(v)>0 $ for all $ v \in \Z^r $, $ v \neq 0 $.

	 We attach the \emph{Gram matrix} to $ Q $, given by
\begin{equation}\label{eq:gram}
	M = M_Q =
\begin{pmatrix}
		a_{11}&\frac{1}{2}a_{12}&\cdots&\frac{1}{2}a_{1r}\\
		\frac{1}{2}a_{12}&a_{22}&\cdots&\frac{1}{2}a_{2r}\\
		\vdots&\vdots&\ddots&\vdots\\
		\frac{1}{2}a_{1r}&\frac{1}{2}a_{2r}&\cdots&a_{rr}
\end{pmatrix}
.
\end{equation}
	Taking $ v \in \Z^r $ to be a~column vector $ (x_1, \dots, x_r)^t $ we have
\[
	Q(v) = v^t Mv.	
\]

	The quadratic form $ Q $ is called \emph{classical} if all the entries of $ M $ are integers, i.e., if $ a_{ij} $ are even for all $ i \neq j $.

	Each quadratic form has an associated \emph{bilinear form} $ B $ defined by
\[
	Q(v+w) = Q(v)+Q(w)+2B(v,w),\qquad v, w \in \Z^r.
\]

	A positive definite form satisfies the Cauchy--Schwarz inequality: For all $ v $ and $ w $,
\[
	Q(v)Q(w) \geq B(v,w)^2.	
\]

	In the 90's, Conway, Miller, Schneeberger, and Simons, and then Bhargava and Hanke \cite{Bha, BH} came up with the following fascinating criteria for universality.
	
\begin{theorem}
		Let $Q$ be a~positive definite quadratic form over $\Z$. Then:
\begin{enumerate}[(a)]
\item  (15-Theorem, Conway--Schneeberger, $ \sim 1995 $)
			If $Q$ is classical and represents the integers	
\[
			1, 2, 3, 5, 6, 7, 10, 14\text{, and }15,			
\]
then it is universal.
			\item (290-Theorem, Bhargava--Hanke, $ \sim 2005 $ \textrm{\cite{BH}})
			If $Q$ represents the integers
\begin{gather*}
				1, 2, 3, 5, 6, 7, 10, 13, 14, 15, 17, 19, 21, 22, 23, 26,\\
				29, 30,
				31, 34, 35, 37, 42, 58, 93, 110, 145, 203\text{, and }290,
\end{gather*}
then it is universal.
			\item Both of these lists of integers are minimal in the sense that for each integer $n$ in the list, there exists a~corresponding quadratic form that represents all of $\mathbb Z_{>0}\setminus \{n\}$, but does not represent $n$.	
\end{enumerate}
	
\end{theorem}

	While the 15-Theorem in part~(a) is not too hard to prove, the 290-Theorem in part~(b) is very challenging, not only because of the large amount of computations needed.

	There have been a~number of further exciting developments related to universal quad\-ratic forms over $\Z$. For example, the conjectural 451-Theorem by Rouse~\cite{Ro} says that \textit{if a~positive definite form represents the integers $1,3,5,\dots,451$, then it represents all \emph{odd} positive integers}. This result has been proved only under the assumption that each of the ternary forms $x^2 + 2y^2 + 5z^2 + xz, 	x^2 + 3y^2 + 6z^2 + xy + 2yz, 	x^2 + 3y^2 + 7z^2 + xy + xz$ represents all odd positive integers (that seems very hard to establish).

	\subsection*{Escalations}
	We give a~sketch of Bhargava's proof of the 15-Theorem. The idea is to ``build up" a~universal quadratic form $ Q $ by gradually adding variables.

	In order for $ Q $ to be universal, it must represent 1, and so it must contain $ x^2 $. Slightly more precisely, a~linear change of variables does not change universality and gives us $ x^2 $ (this will be made more precise soon, once we discuss quadratic lattices).

	Now $ x^2 $ is clearly not universal as it does not represent 2, hence $ Q $ must contain $ 2y^2 $ (again after a~change of variables). We get the form $ x^2 +2axy+2y^2 $, where the coefficient of $ xy $ is $ 2a $ because we require the form to be classical, and so the corresponding Gram matrix is
\[
\begin{pmatrix}
		1&a\\
		a&2
	
\end{pmatrix}
.
\]

	What are the possible values for $ a~$? By the Cauchy--Schwarz inequality $ 1\cdot 2 \geq a^2 $, which leaves the possibilities $ a~= 0, 1, -1 $ with the corresponding Gram matrices
\[
\begin{pmatrix}
		1&0\\
		0&2
\end{pmatrix}
,
\begin{pmatrix}
		1&1\\
		1&2
\end{pmatrix}
,
\begin{pmatrix}
		1&-1\\
		-1&2
\end{pmatrix}
.
\]

	The quadratic forms $ x^2 +2xy+2y^2 $ and $ x^2 -2xy+2y^2 $ are equivalent by the change of variables $ y \mapsto -y $ so we can forget about the third matrix. As for the second matrix, we can reduce the quadratic form by changing variables:	
\[
	x^2 +2xy+2y^2 = (x+y)^2 +y^2 = X^2 +Y^2.
\]

	Note that, in terms of matrices, the Gram matrix of the resulting form is $ C^t MC $ for an invertible matrix $ C $. It can be obtained from $ M $ by successively applying the same row and column operations:
\[
\begin{pmatrix}
		1&1\\
		1&2
	
\end{pmatrix}
\sim
\begin{pmatrix}
		1&0\\
		1&1
	
\end{pmatrix}
\sim
\begin{pmatrix}
		1&0\\
		0&1
	
\end{pmatrix}
.
\]

	We see that after two steps of escalations, we have two candidate forms $ x^2 +2y^2 $ and $ x^2 +y^2 $. Since they do not represent 5, respectively 3, we pass to the matrices	
\[
\begin{pmatrix}
		1&0&b_1\\
		0&2&c_1\\
		b_1&c_1&5
	
\end{pmatrix}
,
\begin{pmatrix}
		1&0&b_2\\
		0&1&c_2\\
		b_2&c_2&3
	
\end{pmatrix}
.
\]

	We again determine all possible values for the coefficients and reduce the forms, which leads to the following possibilities (this can be done as an exercise by the reader):
	\begin{align*}
		&\begin{pmatrix}
			1&&\\
			&1&\\
			&&d_1
		\end{pmatrix},&d_1& = 1, 2, 3\\
		&\begin{pmatrix}
			1&&\\
			&2&\\
			&&d_2
		\end{pmatrix},&d_2& = 2, 3, 4, 5\\
		&\begin{pmatrix}
			1&&\\
			&2&1\\
			&1&d_3
		\end{pmatrix},&d_3& = 4, 5.
	\end{align*}
 
	Continuing this process for rank 4, we get 207 forms, 201 of which are universal. This can be proved by local methods and genus theory (i.e., by a~suitable use of the local-global principle). The remaining 6 forms represent all but one integers. After adding one more variable, we get 1630 universal forms of rank~5.

	This procedure showed that if $ Q $ is universal, then it contains one of the rank 4 or 5 forms obtained above. These are all universal, and so the converse implication also holds, i.e., any quadratic form that contains one of these forms is universal.

	But in the process of escalations, we only considered representations of small integers:
\begin{center}
\begin{tabular}
{l|l}
			\hline
			rank&\\
			\hline
			1&1\\
			2&1, 2\\
			3&1, 2, 3, 5\\
			4 \& 5&1, 2, 3, \dots, 15\\
			\hline
\end{tabular}
\end{center}

	Thus if $ Q $ represents the integers 1, 2, 3, \dots, 15, then it is universal, proving the 15-Theorem.

	\subsection*{Proof of 290-Theorem}
	The proof of the 290-Theorem, although similar, is much more complicated. First, there are more cases to be considered (we have to continue the escalations up to rank 7, which leaves us with approximately 20 000 cases). Second, proving universality is sometimes very non-trivial and uses tools such as theta series (modular forms). For more information, see the original papers~\cite{Bha,BH} or the surveys~\cite{Hah, Moo}.

	\section{Quadratic Lattices}\label{sec:3}
	Talking about changes of variables and adding a~variable in each step of the escalation process is unpleasant. A more efficient approach is to work with quadratic lattices. Their theory is extensive, but we will keep it to the necessary minimum and refer the reader to the book by O'Meara~\cite{OM} for more details.

	\subsection*{Abstract lattices}
	Let $ K $ denote a~number field of degree $ d $ over $\Q$, $ \cO_K $ its ring of integers, and $ V $ a~finite dimensional $ K $-vector space. A subset $ L \subset V $ is an \emph{$ \cO_K $-lattice} if it is a~finitely generated $ \cO_K $-module.
	
\begin{example}
		If $ v_1, v_2, \dots, v_r $ is a~basis of $ V $, then $ L = \cO_K v_1+\cdots+ \cO_K v_r $ is the \textit{free} $ \cO_K $-lattice of rank~$ r $.
	
\end{example}

	We can wonder about the converse statement: Is every $ \cO_K $-lattice of this form? The answer is no, as the following theorem shows.
	
\begin{theorem}[{\cite[81:5]
{OM}}]		Let $ L \subset V $ be an $ \cO_K $-lattice. Then there exist linearly independent vectors $ v_1, \dots, v_r $ in $ V $ and a~fractional ideal $ A $ in $ K $ such that
\[
		L = \cO_K v_1+ \cO_K v_2+\dots+ \cO_K v_{r-1}+A v_r.		
\]
		In particular, the lattice $ L $ is not free when $ A $ is not a~principal ideal.
	
\end{theorem}
	Of course, when $ K $ has class number $ h_K = 1 $, then all fractional ideals are principal and all $ \cO_K $-lattices are free.

	When $L$ is as in theorem above, then $r$ is its \textit{rank}.

	\subsection*{Quadratic lattices}
	Recall that $ V $ denotes a~finite dimensional $ K $-vector space. We moreover assume that $ V $ is a~\emph{quadratic space}, i.e., we a~have quadratic form $ Q: V \to K $ and the attached bilinear form
\[
	B(v, w) = \frac{1}{2}\left(Q(v+w)-Q(v)-Q(w)\right).	
\]

	If $ L \subset V $ is an $ \cO_K $-lattice, then the pair $ (L, Q) $ is called a~\emph{quadratic $ \cO_K $-lattice} (more precisely, we should take the restriction $ Q|_L $ instead of $ Q $).

	A quadratic $ \cO_K $-lattice $ (L, Q) $ is called \emph{integral} if $ Q(v) \in \cO_K $ for all $ v \in L $. In this case $ B(v, w) \in \frac{1}{2} \cO_K $ for all $ v, w \in L $. If $ B(v, w) \in \cO_K $ for all $ v, w \in L $, we say that the quadratic lattice is \emph{classical}.

	The language of quadratic lattices lets us make some of the arguments of the preceding section on escalations more formal. Let $ (L, Q) $ be a~$ \Z $-lattice. Instead of ``$ Q $ contains $ x^2 $" we can say ``there exists $ v_1 \in L $ such that $ Q(v_1) = 1 $", instead of ``$ Q $ must represent 2" we would say ``there exists $ v_2 \in L $ such that $ Q(v_2) = 2 $. What are now the possibilities for $ B(v_1, v_2) $?" and so on.

	Nevertheless, we will sometimes still just talk about quadratic forms with the understanding that the discussion can be made more precise in the language of lattices.

	Specifically, when we have a~free lattice $\cO_K^r$, then the corresponding quadratic form looks like~\eqref{eq:form}, except that now $a_{ij}\in\cO_K$, and we again have the Gram matrix given by~\eqref{eq:gram}.

	\

	As before, we want to study positive definite quadratic forms, but now over a~number field $ K $.

	First, note that if $ [K:\Q] = d $, then there are $ d $ embeddings $ \sigma: K \hookrightarrow \mathbb C $. We say that $ K $ is \emph{totally real} if $ \sigma(K) \subset \R $ for every $ \sigma $. Concretely, let's take $ K = \Q(\gamma) $ for an algebraic integer $ \gamma $ whose minimal polynomial is
\[
	f(X) = X^d +a_1X^{d-1}+\cdots+a_{d-1}X+a_d,\qquad a_i \in \Z.	
\]
	Let $ \gamma_1,\dots,\gamma_d $ denote the complex roots of $ f(X) $. Each embedding $ \sigma_i: K \hookrightarrow \mathbb C $ is completely determined by the image of $ \gamma $, which is sent to one of its conjugates, so that after relabeling $ \sigma_i(\gamma) = \gamma_i $. We see that $ K $ is totally real if and only if $ \gamma_i \in \R $ for all $ i $.

	From now on let's always assume that the number field $ K $ is totally real.

\medskip

	An element $ \alpha \in K $ is \emph{totally positive} if $ \sigma(\alpha) > 0 $ for all the embeddings $ \sigma $. We write $ \alpha \succ \beta $ for $ \alpha, \beta \in K $ if $ \sigma(\alpha)>\sigma(\beta) $ for all $ \sigma $. In particular $ \alpha \succ 0 $ means that $ \alpha $ is totally positive. The totally positive elements $ \alpha \in \cO_K $ form a~semiring which we denote $ \cO_K^+ $.

	The quadratic lattice $ (L, Q) $ is \emph{totally positive definite} if $ Q(v) \succ 0 $ for all $ v \in L\setminus\{0\} $.
	
\begin{definition}
		Let $ (L, Q) $ be an integral totally positive definite lattice. We say that $ Q $ is \emph{universal} if it represents all totally positive integers, i.e., if $ \forall \alpha \in \cO_K^+ \; \exists v \in L: Q(v) = \alpha $.
	
\end{definition}

	\subsection*{Sums of lattices}
	If $ V $ is a~finite dimensional $ K $-vector space and $ L_1, L_2 \subset V $ are two $ \cO_K $-lattices, we define their \emph{sum} as
\[
	L_1+L_2 = \{v+w\mid v \in L_1, w \in L_2\}.	
\]
	It is a~\emph{direct sum} if $ L_1 \cap L_2 = 0 $, in which case we write $ L_1 \oplus L_2 = L_1+L_2 $.

	Assuming that $ V $ is a~quadratic space with a~quadratic form $ Q $, the sum of $ L_1 $ and $ L_2 $ is \emph{orthogonal}, denoted $ L_1 \perp L_2 = L_1+L_2 $, if $ B(v, w) = 0 $ for any $ v \in L_1 $ and $ w \in L_2 $.

	We further use the following notation for \emph{diagonal} quadratic lattices: $ \langle a\rangle $ is the rank one lattice $ \cO_K v $ with $ v\in V $ such that $ Q(v) = a$. Then we define
\[
	\langle a_1, a_2, \dots, a_r \rangle = \langle a_1 \rangle \perp \cdots \perp \langle a_r \rangle = \cO_K v_1 \perp \cO_K v_2 \perp \cdots\perp \cO_K v_r,	
\]
	where $ v_i \in V $ satisfies $ Q(v_i) = a_i $ (and $ B(v_i, v_j) = 0 $ for $ i \neq j $).

	We next state a~useful proposition on the ``splitting off units" that is proved quite easily using Gram--Schmidt orthogonalization.

\begin{proposition}\label{pr:split}
		\label{propSplit}
		Let $ (L, Q) $ be a~\emph{classical} $ \cO_K $-lattice, and let $ v \in L $ be such that $ Q(v) = \varepsilon $ is a~unit in $ \cO_K$. Then $ L = \langle \varepsilon \rangle \perp L' $ for some lattice $ L' \subset L $.
	
\end{proposition}

\begin{proof}
		Let $ L' = (\cO_K v)^{\perp} = \{w \in L \mid B(w, v) = 0\} $. Clearly $ \cO_K v = \langle \varepsilon \rangle $ is orthogonal to $ L' $ and we must show that $ L = \cO_K v + L' $.

		Take any $ z \in L $ and set $ w := z-B(z, v)\varepsilon^{-1}v $. The vector $ B(z, v)\varepsilon^{-1}v $ belongs to $ \cO_K v$, because $ \varepsilon $ is invertible and $B(z, v)\in\cO_K$ as $L$ is classical. And we have $ w \in L' $ because
\[
		B(w, v) = B(z-B(z, v)\varepsilon^{-1}v, v) = B(z, v)-B(z, v)\varepsilon^{-1}B(v, v) = 0,		
\]
		where we used $ B(v, v) = Q(v) = \varepsilon $.
		Thus $z=B(z, v)\varepsilon^{-1}v+w\in \cO_K v + L' $.
\end{proof}

Before turning to the recent developments on universal forms that constitute our main topic, let's briefly comment on three closely related fields of interest:

Indefinite quadratic forms (and forms over number fields that are not totally real) behave quite differently from our case.
 Let's only briefly remark that, for example,~\cite{Si2} and~\cite{EH} characterized general number fields with universal sums of five and three squares, and then~\cite{HHX, HSX, XZ} considered more general universal forms.
Another interesting topic is the study of universal Hermitian quadratic forms over imaginary quadratic fields, e.g.,~\cite{EK2,KP2}.

Regular quadratic forms are forms that represent all elements that are not ruled out by local obstructions~\cite{CI,Ea}, and thus present a~natural generalization of universal forms.
Their theory in many aspects parallels the theory of universal forms; in fact, tools such as Watson's transformations~\cite{CE+,CEO,Wa} sometimes allow one to convert regular forms into universal ones (although special care must be paid to what happens dyadically, as well as over number fields with non-trivial narrow ideal class group).

In connection to this, let's also briefly mention the recent computational results by Kirschmer and Lorch~\cite{Kir,LK} that classify 1-class genera of quadratic lattices over number fields.

Further, let's note that besides from studying representations of integers by quadratic forms, there have been numerous works considering representations of quadratic forms by quadratic forms and, in particular, by the sum of squares, e.g.,~\cite{BI, BC+, Ic, JKO, KO1, KO2, KO3, Ko, KrY, Mo1, Mo2, Oh, Sa1}. Most of them deal with forms over $\Z$, but it is another exciting direction of future research to consider the situation over number fields in detail.

	\section{Real Quadratic Fields}\label{sec:4}
	In this section, let's consider universal forms over a~real quadratic field $ K = \Q(\sqrt{D}) $ with squarefree $ D > 1 $. For simplicity, we always assume $ D \equiv 2, 3 \pmod{4} $ so that $ \cO_K = \Z[\sqrt{D}] $ (but everything that we discuss here also generalizes to the case $D\equiv 1\pmod 4$).

	There are two embeddings $ K \hookrightarrow \mathbb R $, the identity and	
\[
	\alpha = a+b\sqrt{D} \mapsto \alpha' = a-b\sqrt{D}.
\]
	Thus $ \alpha $ is totally positive if and only if $ a+b\sqrt{D}>0 $ and $ a-b\sqrt{D}>0 $. The norm of $ \alpha $ is $ N(\alpha) = \alpha \alpha' = a^2 -b^2 D $, and its trace is $\Tr(\alpha)=\alpha+\alpha'=2a$.

	\subsection*{Diagonal forms and indecomposables}
	An easy example of a~quadratic lattice is $ (\cO_K^r, Q) $, where
\[
	Q(x_1, \dots, x_r) = a_1 x_1^2 +\cdots+a_r x_r^2	
\]
	is a~\textit{diagonal} form. The lattice is integral and totally positive if and only if all the coefficients $ a_i\in \cO_K^+ $.

\medskip

	A key tool for working with (diagonal) universal forms is the notion of an indecomposable element:

	We say that $ \alpha \in \cO_K^+ $ is \emph{indecomposable} if $ \alpha \neq \beta+\gamma$ for $ \beta, \gamma \in \cO_K^+ $.

	To explain the connection, let's assume that a~diagonal quadratic form $ Q $ is universal. Then we can express any indecomposable $ \alpha $ as
\[
	\alpha = a_1 v_1^2 +\cdots+a_r v_r^2,	
\]
	and hence $ \alpha = a_iv_i^2 $ for some $ i $ thanks to indecomposability. Thus each indecomposable essentially appears as a~coefficient in $ Q $, and we can conclude that the number of variables $ r $ of a~universal quadratic form is bounded from below by the number of square classes of indecomposables.

	The key question that remains to be answered is: Are there any indecomposables? Luckily, yes:

\begin{lemma}
		If $ \varepsilon $ is a~totally positive unit, then it is indecomposable.
	
\end{lemma}

\begin{proof}
		Suppose for contradiction that $ \varepsilon = \beta+\gamma $ for some $ \beta, \gamma \in \cO_K^+ $. Then
\[
		1 = N(\varepsilon) = (\beta+\gamma)(\beta'+\gamma') = N(\beta)+N(\gamma)+\beta\gamma'+\beta'\gamma > N(\beta)+N(\gamma) \geq 2.\qedhere	
\]
\end{proof}

Essentially the same proof also works in a~totally real field of a~general degree $d$ and shows that each element of norm $<2^d$ is indecomposable.

Brunotte~\cite{Br1, Br2} gave a~general upper bound on the norm of an indecomposable integer, and so in each $ K $, there are finitely many indecomposables up to multiplication by totally positive units. Unfortunately, the bound is exponential in the regulator of the number field, and so it is not very useful. While this can be significantly improved (see Theorem~\ref{th 6.1} below), it is still important to obtain more information about indecomposables, ideally in the form of an explicit construction. In real quadratic fields, this is possible using continued fractions, as we shall see next.

	\subsection*{Continued fractions}
	The fundamental unit of a~real quadratic field can be given in terms of the continued fraction of $ \sqrt{D} $. It is periodic
\[
	\sqrt{D} = [u_0, \overline{u_1, \dots, u_s}] = [u_0, {u_1, \dots, u_s}, u_1, \dots, u_s, u_1, \dots] = u_0+\frac{1}{u_1+\frac{1}{u_2+\cdots}},
\]
	and we know that $ u_0 = \lfloor \sqrt{D} \rfloor $ and $ u_s = 2\lfloor \sqrt{D} \rfloor $.

	Let
\[
	\frac{p_i}{q_i} = [u_0, \dots, u_i]
\]
	be the \emph{convergents} of the continued fraction. These give the good approximations to $ \sqrt{D} $ since	
\[
	\left|\frac{p_i}{q_i}-\sqrt{D}\right| < \frac{1}{u_{i+1}q_i^2}.
\]

	By an abuse of terminology, the quadratic integers $ \alpha_i = p_i+q_i\sqrt{D} $ will also be called \emph{convergents}. The element $ \alpha_{s-1} $ is the fundamental unit. In other words, it generates the group of units, which can be described as	
\[
	 \cO_K^\times = \left\{\pm \alpha_{s-1}^k \mid k\in\Z\right\}.
\]

	When are the convergents $ \alpha_i $ totally positive? We always have $ \alpha_i > 0 $, and it turns out that $ \alpha_i' > 0 $ if and only if $ i $ is odd.

	Consequently, the fundamental unit $ \alpha_{s-1} $ is totally positive if and only if $ s $ is even. {Thus for $ s $ even, the Pell equation $ x^2 -Dy^2 =-1 $ has no solutions. For $ s $ odd, it has a~solution, namely $ \alpha_{s-1} = x+y\sqrt{D} $.} When $ s $ is odd, the smallest totally positive unit (greater than 1) is then $ \alpha_{2s-1} = \alpha_{s-1}^2 $.

\medskip

	The convergents $ \alpha_i $ satisfy the recurrence
\[
	\alpha_{i+1} = u_{i+1}\alpha_i+\alpha_{i-1},\qquad i \geq 0,
\]
	with $ \alpha_{-1} = 1 $, $ \alpha_0 = \lfloor \sqrt{D} \rfloor + \sqrt{D} $. We observe that multiplication by $ \alpha_{s-1} $ shifts indices: $ \alpha_i \alpha_{s-1} = \alpha_{i+s} $.

	The convergents of a~continued fraction are characterized by their best approximation property. One could say that being indecomposable is a~form of ``totally positive best approximation property", and so the following classical theorem~\cite{DS} should not be too surprising.

\begin{theorem}[{\cite[Theorems 2 and 3]
{DS}}]		The indecomposables $ \alpha $ are precisely the \emph{semiconvergents}, i.e., elements of the form 
\[
		\alpha = \alpha_{i, t} = \alpha_i+t\alpha_{i+1},\qquad i \geq -1\text{ odd},\ 0 \leq t < u_{i+2},	
\]
		and their conjugates. Moreover, $ N(\alpha) \leq D $ for every indecomposable $\alpha$.
	
\end{theorem}
	Note that $ \alpha_{i, u_{i+2}} = \alpha_i+u_{i+2}\alpha_{i+1} = \alpha_{i+2} $, so indecomposables with a~fixed $ i $ form an arithmetic progression going from $ \alpha_i $ to $ \alpha_{i+2} $.

	There have been several improvements on the upper bound for the norm of indecomposables~\cite{JK, Ka2, TV}, and on the additive structure of quadratic fields~\cite{HK,KZ}.

\begin{example}
		Let's find all indecomposables for $ D = 6 $. The convergents of the continued fraction expansion $ \sqrt{6} = [2, \overline{2, 4}] $ are
\begin{align*}
			\alpha_{-1} & = 1\\
			\alpha_0 &  = 2+\sqrt{6}, &\frac{p_0}{q_0} & = 2\\
			\alpha_1 & = 5+2\sqrt{6}, &\frac{p_1}{q_1} & = 2+\frac{1}{2} = \frac{5}{2}.		
\end{align*}

		The element $ \alpha_1 $ is a~totally positive unit. Thus all indecomposables up to multiplication by $ \alpha_1 $ are $ \alpha_{-1, t} $ for $ 0 \leq t < u_1 = 2 $. We have

\[
		\alpha_{-1, 0} = 1,\qquad \alpha_{-1, 1} = 1+(2+\sqrt{6}) = 3+\sqrt{6}
\]
		with $ N(3+\sqrt{6}) = 3 $, and
\[
		\alpha_{1, 0} = \alpha_1 = 5+2\sqrt{6},\qquad \alpha_{1, 1} = \alpha_1\cdot \alpha_{-1, 1} = 27+11\sqrt{6}.
\]

		The semiconvergents $ \alpha_{1, t} = \alpha_1\cdot \alpha_{-1, t} $, $ \alpha_{3, t} = \alpha_1^2 \cdot \alpha_{-1, t} $ etc. differ from $ \alpha_{-1, t} $ by a~multiple of a~unit but in the case of $ \alpha_{1, t} $ not by a~multiple of a~\textit{square} of a~unit.

		In conclusion, there are four square classes of indecomposables represented by $ 1 $, $ 3+\sqrt{6} $, $ 5+2\sqrt{6} $, and $ 27+11\sqrt{6} $.

\medskip

		Now we apply our findings to the problem of universality. Let $ Q $ be a~diagonal universal quadratic form. Since it must represent $ 1 $ and $ 3+\sqrt{6} $, it contains
\begin{equation}\label{eqQuadForm1}
			1\cdot x^2 +(3+\sqrt{6})\cdot y^2.
\end{equation}
		The totally positive unit $ 5+2\sqrt{6} $ is represented by $ Q $, but it is not a~square and $ N(5+2\sqrt{6}) = 1 $, so it is not represented by~(\ref{eqQuadForm1}). Hence $ Q $ must contain
\begin{equation}\label{eqQuadForm2}
			1\cdot x^2 +(3+\sqrt{6})\cdot y^2 +(5+2\sqrt{6})\cdot z^2.
\end{equation}
		The indecomposable $ 27+11\sqrt{6} $ has norm 3 but its square class is not represented by $ 3+\sqrt{6} $, and therefore $ Q $ contains
\begin{equation}\label{eqQuadForm3}
			1\cdot x^2 +(3+\sqrt{6})\cdot y^2 +(5+2\sqrt{6})\cdot z^2 +(27+11\sqrt{6})\cdot w^2.
\end{equation}
		This shows that each diagonal universal quadratic form must have rank $ r \geq 4 $.
		(Of course, the preceding argument could be more formally stated in the language of quadratic lattices.)
	
\end{example}

	\subsection*{Construction of a~universal form}
\begin{lemma}\label{observationSumOfIndec}
	Every $ \alpha \in \cO_K^+ $ is a~sum of indecomposables.
\end{lemma}

\begin{proof}
		If $ \alpha $ is not indecomposable, then $ \alpha = \beta+\gamma $. But then $ \Tr(\alpha) = \Tr(\beta)+\Tr(\gamma) $ and the traces are positive integers (because the elements are totally positive and integral). Therefore $ \Tr(\beta), \Tr(\gamma) < \Tr(\alpha) $. The result follows by induction.
\end{proof}

	Let $ \varepsilon $ be the totally positive fundamental unit; in other words, a~generator of the group of all totally positive units
\[
	 \cO_K^{\times,+} = \{\varepsilon^\ell\mid\ell\in \Z\}.
\]
	 We have seen that $ \varepsilon $ equals $\alpha_{s-1} $ or $ \alpha_{2s-1} $ depending on whether $ s $ is even or odd, respectively.

	 Further, let $ S $ be the set of representatives of indecomposables up to multiplication by $ \cO_K^{\times,+} $. We can take
\[
	S = \{\alpha_{i, t_i}\mid i = -1, 1, 3, \dots, k, \; 0 \leq t_i < u_{i+2}\},
\]
	where $ k = s-3 $ if $ s $ is even and $ k = 2s-3 $ if $ s $ is odd. In particular, the number of elements in $ S $ is
\begin{align*}
	\#S &= u_1+u_3+u_5+\cdots \\
        &=
        \begin{cases}
		u_1+u_3+\cdots+u_{s-3}+u_{s-1},&s\text{ even},\\
		u_1+u_3+\cdots+u_{2s-3}+u_{2s-1}=u_1+u_2+u_3+\cdots+u_{s-1}+u_{s},&s\text{ odd}.
\end{cases}
\end{align*}

	Lemma~\ref{observationSumOfIndec} tells us that any totally positive $ \alpha $ is a~sum of indecomposables. We group the indecomposables according to their class in $ S $ to express $ \alpha $ as
\begin{equation}\label{eqAlpha}
		\alpha = \sum_{\sigma \in S}\sigma e_\sigma,
\end{equation}
	where each $ e_\sigma $ is a~linear combination of totally positive units with non-negative coefficients.

\begin{lemma}
		For each $ e_\sigma $, there exist $ j\in \Z $ and integers $ c, d \geq 0 $ such that
\[
		e_\sigma = c\varepsilon^j +d\varepsilon^{j+1}.
\]
\end{lemma}

	The idea of the proof is to rewrite the linear combination $ e_\sigma $ using the minimal polynomial $ \varepsilon^2 -n \varepsilon+1 = 0 $, one just needs to arrange things so that $c,d$ are indeed non-negative.

\begin{theorem}[{\cite[Theorem 1]
{Ki2},~\cite[Theorem 10]{BK2}}]\label{th:construct}
		The quadratic form
\[
		{\displaystyle\bigperp_{\sigma \in S}}\langle \sigma, \sigma, \sigma, \sigma, \varepsilon\sigma, \varepsilon\sigma, \varepsilon\sigma, \varepsilon\sigma \rangle
\]
		is universal and has $ 8\cdot \#S $ variables. (Here $\textstyle\bigperp$ denotes an orthogonal sum of the diagonal lattices.)
	
\end{theorem}

\begin{proof}
		Let $ \alpha $ be a~totally positive integer and write it as in~(\ref{eqAlpha}). By the preceding lemma, each coefficient $ e_\sigma $ is of the form $ e_\sigma = c\varepsilon^j +d\varepsilon^{j+1} $. All it remains to show is that $ e_\sigma $ is represented by the form		
\[
		\langle 1, 1, 1, 1, \varepsilon, \varepsilon, \varepsilon, \varepsilon \rangle = x_1^2 +x_2^2 +x_3^2 +x_4^2 +\varepsilon\cdot(x_5^2 +x_6^2 +x_7^2 +x_8^2).
\]
		When $ j $ is even, then $ c $ is represented by $ x_1^2 +x_2^2 +x_3^2 +x_4^2 $ by Lagrange's Four Square Theorem, and $ c\varepsilon^j $ also, and the second term $ \varepsilon\cdot d\varepsilon^j $ is represented by $ \varepsilon\cdot(x_5^2 +x_6^2 +x_7^2 +x_8^2) $. The case of odd $j$ is very similar.
\end{proof}

	\subsection*{Sums of continued fraction coefficients}
	How big is $ S $? We would like to bound its size in terms of $ D $, $ \varepsilon $, and the class number $ h_D $.

	In the case of odd $s$ we have the trivial bound
\[
	u_1+u_2+\cdots+u_s \geq u_s = 2\lfloor \sqrt{D} \rfloor > \sqrt{D}.
\]
	On the other hand, we know the following:

\begin{theorem}[{\cite[Corollary 18]
{BK2}}]		There is a~positive constant $ c $ such that
\[
		u_1+u_2+\cdots+u_s \leq c\sqrt{D}(\log D)^2.
\]
\end{theorem}
	Roughly speaking, the preceding estimate (which can be somewhat improved) comes from the class number formula
\[
	h_D = \frac{\sqrt{D}}{\log \varepsilon}L(1, \chi).
\]
	It is know that $ L(1, \chi) \ll \log D $, and since clearly $ h_D \geq 1 $,
\[
	\log \varepsilon \ll \sqrt{D}\log D.
\]
	Here
	we use the usual analytic notations
	that $f\ll g$ (and $g\gg f$) if there is a~constant $C>0$ such that $f(x)<Cg(x)$ for all $x$ (that lie in the domains of $f,g$).
	We use $f\ll_P g$ or $g\gg_P f$ to stress that the constant $C$ depends on the specified parameter(s) $P$.

\medskip

	Next we relate $ \log \varepsilon $ to the sum $ \sum_{i=1}^s u_i $. We have
\[
	\varepsilon = \alpha_{s-1} = u_{s-1}\alpha_{s-2}+\alpha_{s-3} \geq u_{s-1}\alpha_{s-2} \geq u_{s-1}u_{s-2}\alpha_{s-3} \geq \dots \geq u_{s-1}u_{s-2}\cdots u_0,
\]
	hence $ \log \varepsilon \geq \sum_{i=1}^s \log u_i $. This at least proves that $ s \ll \sqrt{D}\log D $ (if $ u_i = 1 $ for some $ i $, we have to proceed more carefully). See also~\cite{KM} for a~more detailed discussion.
\medskip

	In the case when $ s $ is even and $ \varepsilon = \alpha_{s-1} $, it is quite subtle trying to estimate $ u_1+u_3+\dots+u_{s-1} $ because it can be small: e.g., for the continued fraction
\[
	\sqrt{n^2 -1} = [n-1, \overline{1, 2(n-1)}],
\]
	$ S $ contains only one element even though $ D = n^2 -1 $ grows to infinity.

	\section{(Non-)Existence of Universal Forms}\label{sec:5}
	Let's now turn our attention more generally to the questions of existence of universal forms and of their properties, such as possible ranks.	(Again, our discussion always applies to quadratic lattices, even when we talk about quadratic forms.)

	We will still (mostly) treat the case of real quadratic fields $ K = \Q(\sqrt{D}) $ in this section. Above, we constructed a~universal form over every such $ K $ (with $ D \equiv 2, 3 \pmod 4 $, although this assumption is not necessary). In fact, a~universal form exists in every number field, and there are (at least) two ways of proving this:
\begin{enumerate}[a)] 
		\item Proceed similarly as in the case of $ \Q(\sqrt{D}) $ (see Corollary~\ref{cor:6.2} below).
		\item Hsia--Kitaoka--Kneser~\cite[Theorem 3]{HKK} showed a~local-global principle for representations of elements with sufficiently large norm by $ Q $, provided that the rank of $ Q $ is at least 5. So one can:
\begin{itemize}
			\item Find a~form $Q_0$ that represents everything locally over all the finite places. For example, $Q_0=\langle 1,1,1,\alpha\rangle$ where $\alpha\succ 0$ has additive valuation 1 at each dyadic place works, for already $\langle 1,1,1\rangle$ is locally universal at all non-dyadic places~\cite[92:1b]{OM}, and at the dyadic places, one can use Beli's theorem~\cite[Theorem 2.1]{Bel}. Alternatively, one can use Riehm's (much older) theorem~\cite[Theorem 7.4]{Ri} thanks to which it suffices to make sure that all classes mod 2 are represented -- which is easily arranged by adding extra variables.
			\item If necessary, add variables to $Q_0$ to obtain $Q$ of rank $\geq 5$, for which one can use the asymptotic local-global principle~\cite[Theorem 3]{HKK}.
			\item Finally, add extra variables to cover the (finitely many) square classes of elements of small norms that are not represented by $Q$.
		
\end{itemize}
	
\end{enumerate}

{It is easy to see that there is never a~universal form of rank $ r = 1 $ or $ 2 $ (for local reasons). Moreover, when the degree $ d $ of $ K $ is odd, it quickly follows from Hilbert's reciprocity law that there is no ternary universal form~\cite[Lemma 3]{EK1}.}

\medskip

The most natural candidate for a~universal form would be the sum of squares. Unfortunately, it is almost never universal, for Siegel~\cite{Si2} showed that
		a sum of squares is universal over $ \cO_K $ only for
		
\begin{itemize}
			\item 			$ K = \Q $ \ \ \ \ \ \ (when 4 squares suffice) and
			\item $ K = \Q(\sqrt{5}) $ (when 3 squares suffice~\cite{Ma}).
		
\end{itemize}
	The proof considers representations of units and indecomposables and is sketched below as Theorem~\ref{th:8.1}.

	One thus has to consider more general quadratic forms and aim at various classification results. This has been the most successful in the quadratic case.

\begin{theorem}[{\cite[Theorem 1.1]
{CKR}}]		If $ K = \Q(\sqrt{D}) $ has a~ternary classical universal form, then $ D =2, 3 $, or $5$. In total, there are 11 such forms; examples in the three cases are
\begin{itemize}
			\item $ x^2 +y^2 +(2+\sqrt{2})z^2 $ \ for $ D = 2 $,
			\item $ x^2 +y^2 +(2+\sqrt{3})z^2 $ \ for $ D = 3 $,
			\item $ x^2 +y^2 +\frac{5+\sqrt{5}}{2}z^2 $\ \ \ \ \ \ for $ D = 5 $.
\end{itemize}
\end{theorem}

The best available result in this direction is:

\begin{theorem}[{\cite[Theorem 3.2]
{KP2}}]		If $ K = \Q(\sqrt{D}) $ has a~universal lattice of rank $\leq 7$ (and $D$ is squarefree), then $ D< (576283867731072000000005)^2$.
	
\end{theorem}

	This result builds on~\cite{Ki1}; in fact, Kim--Kim--Park~\cite{KP2} give more precise results, also separately for classical lattices. The proof is based on considering the sublattice representing $1,2,\dots, 290$ (it must have rank at least 4 when $D$ is large thanks to the 290-Theorem), and the sublattice representing 
 \[
 \lceil1\cdot\sqrt D\rceil +1\cdot\sqrt D, \lceil2\cdot\sqrt D\rceil +2\cdot\sqrt D, \dots, \lceil290\cdot\sqrt D\rceil +290\cdot\sqrt D
 \]
 (that also must have rank $\geq 4$).

	Note that there is an 8-ary universal form over each $\Q(\sqrt{n^2 -1})$ (when $n^2 -1$ is squarefree)~\cite{Ki2} that is constructed precisely as in Theorem~\ref{th:construct}.

\medskip
	Such results on determining the small possible ranks of universal lattices are motivated by Kitaoka's conjecture.

\begin{conjecture}[Kitaoka]
		There are only finitely many totally real number fields $ K $ having a~ternary universal form.
	
\end{conjecture}

	The conjecture still remains open. However, B. M. Kim and Kala--Yatsyna~\cite{KY3} proved at least a~weak version of the conjecture saying that \textit{when the degree $d$ of $K$ is fixed, then there are only finitely many such fields $K$.}

	Some further interesting results are~\cite{CL+, De1, De2, Le, KTZ, Sa2}.

	\subsection*{Lower bounds on ranks}

	Surprisingly, it turns out that universal lattices can require arbitrarily large ranks.

\begin{theorem}[{\cite[Theorem 1]
{BK1},~\cite[Theorem 1.1]{Ka1}}]		For any positive integer $ r $, there are infinitely many quadratic fields $ \Q(\sqrt{D})$ that do \emph{not} have a~universal lattice of rank $ \leq r $.
	
\end{theorem}

	The broad idea behind the proof is the following. In a~universal lattice $ (L, Q) $, construct a~sublattice that must have rank $ \geq r $, for example by arranging it to contain pairwise orthogonal vectors $ v_i $ representing suitable convergents $ \alpha_i $.

	A more precise result in this direction was obtained by Kala--Tinkov\' a~\cite{KT}, with inspiration by earlier results of Yatsyna~\cite{Ya}.

\begin{theorem}[{\cite[Sections 7.1 and 7.3]
{KT}}]		\label{theoremKT}
		Let $ \sqrt{D} = [u_0, \overline{u_1, \dots, u_s}] $ and
		
\[
		U =
\begin{cases}
			\max(u_1, u_3, \dots, u_{s-1}),&\text{if }s\text{ is even},\\
			\sqrt{D},&\text{if }s\text{ is odd}.
\end{cases}
\]
		Let $ Q $ be a~universal quadratic form over $ \Q(\sqrt{D}) $ of rank $ r $.
		
\begin{enumerate}[a)]
			\item If $ Q $ is classical, then $ r \geq U/2 $.
			\item In general, $ r \geq \sqrt{U}/2 $ (assuming $ U \geq 240 $).
\end{enumerate}
\end{theorem}

	To prove the theorem, we want to use minimal vectors in a~quadratic lattice $ (L, Q) $, i.e., nonzero vectors $ v $ such that $ Q(v) $ is minimal. This approach works best over $ \Z $, so we need to obtain a~$ \Z $-lattice. In general, if $ [K:\Q] = d $ and $ L $ is an $ \cO_K $-lattice of rank $ r $, then $ L $ can be naturally viewed as a~$ \Z $-lattice of rank $ rd $. Indeed,
\[
	L = \cO_K v_1+\cdots+ \cO_K v_{r-1}+A v_r
\]
	for some fractional ideal $ A $. Now $ \cO_K $ and $ A $ are isomorphic to $ \Z^d $ as $ \Z $-modules and hence we can identify $ L \simeq \Z^{dr} $ as a~$ \Z $-module.

	We will consider the quadratic form $ \Tr(\delta Q) $ for a~suitable $ \delta $. We choose $ \delta $ to satisfy that
	
\begin{itemize}
		\item $ \delta $ is a~totally positive element (for then $ \Tr(\delta Q) $ is positive definite), and
		\item $ \Tr(\delta Q(v)) \in \Z $ for any $ v \in L $.
	
\end{itemize}

	This naturally leads us to looking at the \emph{codifferent}
\[
	 \cO_K^\vee = \{\delta \in K\mid \forall \alpha \in \cO_K:\;\Tr(\delta\alpha)\in\Z\}.
\]

	If $ \cO_K = \Z[\sqrt{D}] $, then $ \cO_K^\vee = \frac{1}{2\sqrt{D}}\cdot \cO_K $. The inclusion $ \supset $ is easy to prove as
	
\[
	\Tr\left(\frac{1}{2\sqrt{D}}(a+b\sqrt{D})\right) = \Tr\left(\frac{b}{2}+\frac{a}{2D}\sqrt{D}\right) = b
\]
	for $ a, b \in \Z $ (and the other inclusion is not too hard either).

	We next make the following observation: Let $ \alpha \in \cO_K^+ $. If there exists $ \delta \in \cO_K^{\vee, +} $ such that $ \Tr(\delta\alpha) = 1 $, then $ \alpha $ is indecomposable. For if $ \alpha = \beta+\gamma $ for $ \beta, \gamma \in \cO_K^+$, then
\[
	1 = \Tr(\delta\alpha) = \Tr(\delta\beta)+\Tr(\delta\gamma) \geq 2.
\]
	Now we have what we need to prove Theorem~\ref{theoremKT}.

\begin{proof}[Sketch of proof of Theorem $\ref{theoremKT}$]
$~$

		{Step 1.} Let $ U = u_{i+2} $ for some odd $i$ and consider the indecomposables $ \alpha_{i, t} $, $ 0 \leq t < U $. We define $ \delta = -\frac{1}{2\sqrt{D}}\alpha_{i+1}' $. It can be checked directly that $ \delta \in \cO_K^\vee $ and $ \delta $ is totally positive. Next we compute the trace of	
\[
		\delta \alpha_{i, t} = -\frac{1}{2\sqrt{D}}(p_{i+1}-q_{i+1}\sqrt{D})\cdot(p_i+q_i\sqrt{D})-\frac{t\sqrt{D}}{2D}\alpha_{i+1}'\alpha_{i+1}.
\]
		Since $ \alpha_{i+1}'\alpha_{i+1} = N(\alpha_{i+1}) \in \Z $, we have	
\[
		\Tr(\delta \alpha_{i, t}) = p_iq_{i+1}-p_{i+1}q_i = (-1)^{i+1} = 1.	
\]
\medskip

		{Step 2.} Take a~quadratic $ \cO_K $-lattice $ (L, Q) $ representing all the indecomposables $ \alpha_{i, t} $, $ 0 \leq t < U $, so that $ Q(v_t) = \alpha_{i, t} $ for some $ v_t \in L$. Then $ (\Z^{2r}, \Tr(\delta Q)) $ is a~$ \Z $-lattice containing $ 2U $ vectors of length 1, namely $ \pm v_t $, as
\[
		\Tr(\delta Q(\pm v_t)) = \Tr(\delta \alpha_{i, t}) = 1.
\]
		Observe that if $ Q $ is classical, then $ \Tr(\delta Q) $ is also classical. Therefore by repeatedly splitting off 1 (see Proposition~\ref{pr:split}) we get
\[
		\Z^{2r} = \langle 1 \rangle \perp \langle 1 \rangle \perp \cdots \perp \langle 1 \rangle \perp L',
\]
		where $ \langle 1 \rangle $ is repeated $ U $ times in the diagonal part. Thus $ 2r \geq U $.

\medskip
		{Step 3.} If $ Q $ is non-classical, we use known bound on the number of length-one vectors: There are $ \leq \max(r^2,240) $ of them in a~non-classical $ \Z $-lattice of rank $ r $.
\end{proof}

	\subsection*{Summary}
	Denote $ m(K) $ the minimal rank of a~universal $ \cO_K $-lattice over $ K $ and $ m_{class}(K) $ the minimal rank of a~classical universal $ \cO_K $-lattice.

	We saw that for $ K = \Q(\sqrt{D}) $ with $ \sqrt{D} = [u_0, \overline{u_1, \dots, u_s}] $, we have
\begin{align*}
		\frac{1}{2}\max(u_i)^{1/2}& \leq m(K) \leq 8 \sum_{i=1}^s u_i \ll \sqrt{D}(\log D)^2 \\
		\frac{1}{2}\max(u_i)& \leq m_{class}(K) \leq 8 \sum_{i=1}^s u_i \ll \sqrt{D}(\log D)^2
\end{align*}
	If the fundamental unit is not totally positive (i.e., $ s $ odd), this is not too bad: the lower bound is $ D^{1/4} $ and $ D^{1/2} $ for $ m(K) $ and $ m_{class}(K) $, respectively. In the case when the fundamental unit is totally positive (i.e., $ s $ even), there are arbitrarily large differences between the lower and upper bounds, e.g., for $ \sqrt{D} = [u_0, \overline{1, 1, \dots, 1, 2u_0} ] $, we get $ 1/2 \leq m_{class}(K) \leq 4s $. Obtaining better lower bounds would require including all the indecomposables, not just $ \alpha_{i, t} $ for a~fixed $ i $.

	However, Kala--Yatsyna--\. Zmija recently expanded on these results by showing that
	
\begin{theorem}[{\cite[Theorem 1.1]
{KYZ}}]\label{th1.1}
		Let $\varepsilon >0$.
		For almost all squarefree $D>0$, we have that		
\[
m_{class}(\Q(\sqrt D))\gg_\varepsilon D^{\frac{1}{12} -\varepsilon}\text{\ \ \ and\ \ \ }m(\Q(\sqrt D))\gg_\varepsilon D^{\frac{1}{24} -\varepsilon}.
\]
\end{theorem}

By ``almost all'' we mean that the set of such $D$ has (natural) density 1 among the set of all squarefree $D>0$.

\

	Finally, let's mention an open problem. Thanks to Chan--Oh~\cite{CO}, we know that there exist analogues of the 15- and 290-Theorems over any number field. However, the corresponding bounds are explicitly known only for classical forms over $ \Q(\sqrt{5}) $~\cite{Le}, and determining them more generally seems to be quite hard.

	\section{General Results}\label{sec:6}

	Let's now turn to fields of higher degree, and to several possible approaches for studying universal forms over them. Throughout this section, $ K $ thus denotes a~totally real number field of degree $ d = [K:\Q] $.
	\subsection*{Using units}
	Let $ (L, Q) $ be a~classical universal quadratic lattice. We saw in Proposition~\ref{propSplit} that each unit splits off, and in particular $ L = \langle 1 \rangle \perp L' $ for some lattice $ L' \subset L $. Now any square of a~unit is represented by $ \langle 1 \rangle $, so it need not be represented by $ L' $. But if $ \varepsilon \in \cO_K^{\times, +} $ is a~unit which is not a~square, then $ L' $ must represent $ \varepsilon $ and hence $ L = \langle 1, \varepsilon \rangle \perp L'' $ for some lattice $ L'' \subset L $. Continuing like this leads to the following observation: \textit{The rank of a~classical universal lattice is always greater than or equal to $ \# \cO_K^{\times, +}/ \cO_K^{\times 2}$.}

	Since $ K $ is totally real, there are $ d-1 $ fundamental units $ \varepsilon_1, \varepsilon_2, \dots, \varepsilon_{d-1} $, which implies $ \cO_K^{\times, +} \simeq \Z^{d-1} $. We can distinguish the two extreme cases:
	
\begin{itemize}
		\item No fundamental unit is totally positive. Then each totally positive unit is a~square, i.e., $ \cO_K^{\times, +} = \cO_K^{\times 2} $.
		\item All fundamental units are totally positive. Then $ \cO_K^{\times 2} \simeq (2\Z)^{d-1} $ and $\# \cO_K^{\times, +}/ \cO_K^{\times 2} = 2^{d-1}.$
\end{itemize}

	In the general situation when $ k $ fundamental units are totally positive, the rank of a~classical universal lattice is $ \geq 2^k $. If $ k \geq 2 $, this proves a~special case of Kitaoka's conjecture, i.e., that $ K $ has no ternary classical form. Not much is thus missing to prove the full conjecture, it would suffice to show the existence of a~few indecomposables!

	As we already mentioned, recall that Hilbert's reciprocity law implies a~theorem of Earnest--Khosravani~\cite{EK1}: If $ d $ is odd, then there is no ternary universal lattice (for local reasons).

	\subsection*{Bounds on indecomposables}
	It is not hard to show a~general upper bound on the norm of an indecomposable (although surprisingly, this bound was discovered only very recently, even though this question is, e.g., formulated as~\cite[Problem 53]{Nar}).

\begin{theorem}[{\cite[Theorem 5]
{KY3}}]\label{th 6.1}
		\label{theoremKY}
		Each indecomposable has norm $ \leq \disc_{K/\Q} $. In fact, if $N(\alpha)> \disc_{K/\Q} $, then $\alpha\succ\beta^2$ for some $\beta\in \cO_K$.
	
\end{theorem}

\begin{proof}
		Let's just sketch the proof. Let $ \alpha $ be an element of norm $ N\alpha > \disc_{K/\Q} $, and let $ \sigma_1(\alpha), \dots, \sigma_d(\alpha) $ be its conjugates. By Minkowski's theorem, for a~sufficiently small $ \varepsilon > 0 $ the box
\[
		\left\{x \in \R^d : |x_i| \leq \sqrt{\sigma_i(\alpha)}-\varepsilon, i = 1, \dots, d\right\}
\]
		 in the Minkowski space contains a~non-zero element $ \beta \in \cO_K $. Then $ \alpha \succ \beta^2 $, and $ \alpha $ is therefore decomposable.
\end{proof}

	Before proceeding further, note that we have already seen~\cite{Si2} that typically not all totally positive integers are sums of squares, but we can ask: What is the smallest integer $P$ such that \textit{if an element is the sum of squares, then it is the sum of at most $P$ squares}? This integer $P$ is called the \textit{Pythagoras number} of the ring $ \cO_K$ and is known to be always finite, but can be arbitrarily large~\cite{Sc2} (cf. also~\cite{Po}).
	However, there is an upper bound for Pythagoras numbers of orders in number fields that depends only on the degree of the number field~\cite[Corollary 3.3]{KY2}.

	In the case of real quadratic number fields $K=\Q(\sqrt D)$ the Pythagoras number is always $\leq 5$, and this bound is sharp~\cite{Pe}. In fact, one can show that $\mathrm{P}(\cO_K)=3$ for $D=2,3,5$~\cite{Coh,Sc1} and determine all $D$ for which $\mathrm{P}(\cO_K)=4$ (as in~\cite{CP}). For some further recent results, see~\cite{KY1,Kr,KRS,Ras,Ti2}.

	Thanks to Theorem~\ref{th 6.1} above, if we have an element of large norm, we can successively subtract squares from it until we are left with something of norm $ \leq \disc_{K/\Q} $. If we then rewrite the sum of squares as the sum of $P$ squares, we obtain the following result:

\begin{corollary}[{\cite[Theorem 6]
{KY3}}]\label{cor:6.2}
		The quadratic form
		
\[
		\sum \alpha x_\alpha^2 +y_1^2 +\cdots+y_P^2,
\]
		where we sum over all square classes of elements $ \alpha \in \cO_K^+ $ with norm $ N\alpha \leq \disc_{K/\Q} $, is universal and has rank $ \ll \disc_{K/\Q}\cdot (\log \disc_{K/\Q})^{d-1} $.
\end{corollary}

	It is also sometimes useful to know that there is a~partial converse of Theorem~\ref{theoremKY}.
	
\begin{theorem}
		Assume that $ K $ is primitive, i.e., there is no field $ \Q \varsubsetneq F \varsubsetneq K $. If $ \alpha\in \cO_K^+ $ has norm $ N\alpha \leq \disc_{K/\Q}^{1/(d^2 -d)} $ and is not divisible by any $n\in\Z_{\geq 2}$, then it is indecomposable.
\end{theorem}
	We again only sketch the proof. As usual, $ \sigma_1, \dots, \sigma_d $ denote the $ d $ embeddings of $ K $ into $ \R $. Suppose that $ \alpha = \beta+\gamma $ for some totally positive integers $ \beta $ and $ \gamma $. Then
\[
	N\alpha = (\sigma_1\beta+\sigma_1\gamma)(\sigma_2\beta+\sigma_2\gamma)\cdots(\sigma_d\beta+\sigma_d\gamma) \geq \Tr(\sigma_1\beta\cdot\sigma_2\gamma\cdot\sigma_3\gamma\cdots\sigma_d\gamma) \gg \disc_{K/\Q}^{1/(d^2 -d)}
\]
	by the Stieltjes--Schur Theorem~\cite[\S 3]{Sch}: If $ \mu \succ 0 $ and $ K = \Q(\mu) $, then $ \Tr(\mu) \gg \disc_{K/\Q}^{1/(d^2 -d)} $, cf.~\cite[Proposition 2]{Ka4}.

	\subsection*{Elements of trace 1}

	In Section~\ref{sec:4} we saw lower bounds for ranks of universal quadratic lattices in terms of elements of trace 1 in the codifferent. Exactly the same result holds in general.

\begin{theorem}[{\cite{Ya},~\cite[Section 7.1]
{KT}}]		Assume that there are $ \beta_1, \dots, \beta_u \in \cO_K^+ $, $ \delta \in \cO_K^{\vee, +} $ such that $ \Tr(\beta_i\delta) = 1 $ for all $ i $. Then
\begin{align*}
			m(K) \geq \frac{u}{d},\qquad m_{class}(K) \geq \frac{\sqrt{u}}{d}.
\end{align*}
	
\end{theorem}

	How to find such elements?
	There is no general way (after all, there may be no totally positive elements in the codifferent that have trace 1), so one may have to rely on explicit constructions (such as in the proof of Theorem~\ref{theoremKT} or~\ref{theoremKT2}).
	However, let's also briefly discuss a~method based on interlacing polynomials and another one which uses the Dedekind zeta function.

\begin{theorem}[{\cite[Theorem 4]
{Ya}}]		Let $ d $ and $ r $ be positive integers. There are only finitely many totally real number fields $ K $ of degree $ d $ such that
		
\begin{itemize}
			\item $ K $ is primitive (it has no proper subfields),
			\item $ K $ is monogenic ($ \cO_K = \Z[\alpha] $ for some algebraic integer $ \alpha $),
			\item $ K $ has units of all signatures (equivalently, $ \cO_K^{\times, +} = \cO_K^{\times 2} $),
			\item $ K $ has a~universal lattice of rank $ r $.
\end{itemize}
	
\end{theorem}

	The proof uses interlacing polynomials. Let
\[
	f(x) = (x-\alpha_1)(x-\alpha_2)\cdots(x-\alpha_d)
\]
	be the minimal polynomial for $ \alpha $, and assume that $ \alpha_1 < \alpha_2 < \dots < \alpha_d $. We say that a~polynomial
	
\[
	g(x) = (x-\beta_1)(x-\beta_2)\cdots(x-\beta_{d-1})
\]
	\emph{interlaces} $ f $ is	
\[
	\alpha_1 < \beta_1 < \alpha_2 < \beta_2 < \dots < \beta_{d-1} < \alpha_d.
\]
	The key fact is that there is a~bijection between such polynomials $ g $ and the set 
 \[
     \{\gamma\in\cO_K^{\vee, +}\mid \Tr\gamma = 1\}. \]
     
	The \emph{Dedekind zeta function} is defined as
\[
	\zeta_K(s) = \sum_{A < \cO_K} \frac{1}{(N A)^s},\qquad \Re s > 1.
\]
	The series converges absolutely for $ \Re s > 1 $ and $ \zeta_K $ has a~meromorphic continuation to the entire complex plane with a~simple pole at $ s = 1 $. It satisfies a~functional equation which relates $ \zeta_K(s) $ to $ \zeta_K(1-s) $.

	For us, the important important fact is that Siegel related the value $ \zeta_K(-1) $ to elements of small trace~\cite{Si1},~\cite[\S 1]{Za}:

\begin{theorem}[Siegel's formula for $ \zeta_K(-1) $ and functional equation]\label{theoremSiegelsFormula}
		Assume that $ K $ is a~totally real field of degree $ d=2, 3, 5, 7 $. Then
\[
		\sum_{\substack{\alpha \in \cO_K^{\vee, +}\\\Tr\alpha=1}}\sigma\left((\alpha)(\cO_K^\vee)^{-1}\right) = \frac{1}{b_d}\left|\disc_{K/\Q}\right|^{3/2}\left(\frac{-1}{4\pi}\right)^d \zeta_K(2)
\]
		for a~suitable $ b_d \in \Q $ (e.g., $ b_2 = \frac{1}{240} $, $ b_3 = -\frac{1}{504} $, \dots). Here		
\[
		\sigma(B) = \sum_{A \mid B}N(A).
\]
	
\end{theorem}

A similar formula holds in each degree $d$, but as the degree grows, it will involve elements of large traces (roughly, of traces up to $d/6$).

\medskip

As a~sample application, let's mention the following result on the \textit{lifting problem} (that will discussed in detail in Section~\ref{sec:8}).

\begin{theorem}[{\cite[Theorem 1.2]
{KY2}}]\label{th:lift}
		If $ K $ is a~totally real number field of degree $ d =2, 3, 4, 5, 7 $ which has
		
\begin{itemize}
			\item principal codifferent ideal, and
			\item a~universal quadratic form with coefficients in $ \Z $,
		
\end{itemize}
		then $ K = \Q(\sqrt{5}) $ or $ K = \Q(\zeta_7+\zeta_7^{-1}) $, where $ \zeta_7 = e^{2\pi i/7} $. The form $ x^2 +y^2 +z^2 $ is universal over $ \Q(\sqrt{5}) $, and $ x^2 +y^2 +z^2 +w^2 +xy+xz+xw $ is universal over $ \Q(\zeta_7+\zeta_7^{-1}) $.
\end{theorem}

	\subsection*{Large ranks} Let's also summarize here the known results on the existence of number fields with large minimal rank $m(K)$. For quadratic fields, we have already seen this in Section~\ref{sec:5}, and in the cubic case, this is originally due to Yatsyna~\cite[Theorem 5]{Ya}, and we will establish this in Section~\ref{sec:7} below.

	A natural idea for extending these results to higher degrees is to start with a~field $K$ with large $m(K)$ and to consider overfields $L\supset K$. This was first carried out for multiquadratic fields~\cite{KS}, and then extended to all fields of degrees divisible by 2 and 3~\cite{Ka4}. Finally, Doležálek~\cite{Do} generalized this method to make it readily applicable to quite general extensions $L\supset K$. So it would (mostly) suffice to prove the existence of large ranks $m(K)$ over fields of prime degrees $\geq 5$ -- which, however, remains open so far.

	\section{Families of Cubic Fields}\label{sec:7}
	Over a~fixed field, one can compute everything explicitly, e.g., there are finitely many totally positive elements $ \alpha $ with norm $ N\alpha \leq \disc $ (up to multiplication by units), and we can check which ones are indecomposable. We can also compute the codifferent and check which elements have trace 1.

	For all fields of a~given degree $ d $, the problem is much harder. We used continued fractions to deal with real quadratic fields $ \Q(\sqrt{D}) $. We might attempt to use generalized continued fractions~\cite{Ber,Scw} for fields of a~higher degree -- they are much worse behaved but there are some ongoing works in connection with the Jacobi--Perron algorithm~\cite{KRSt,KST}. Geometric generalized continued fractions~\cite{Kar} are also promising, since there is a~close connection to indecomposables.

	Rather than working with all fields, it is however typically easier to focus on a~suitable family of fields that share some relevant properties (such as the structure of units and indecomposables).

	\subsection*{The simplest cubic fields}
	We describe first the family of totally real cubic fields introduced by Shanks~\cite{Sh}.

	Let $ K = \Q(\rho) $ where $ \rho $ is a~root of the polynomial
\[
	f(x) = x^3 -ax^2 -(a+3)x-1,\qquad a~\geq -1.
\]
	If we order the three roots $ \rho $, $ \rho' $, $ \rho'' $ as $ \rho > \rho'' > \rho' $, then they are of approximate sizes $ \rho \approx a+1 $, $ \rho'' \approx 0 $ and $ \rho' \approx -1 $. It is a~useful fact that all the roots are units, and are permuted under the mapping $ \alpha \mapsto \frac{-1}{1+\alpha} $. We thus see that the other two conjugates $ \rho' $ and $ \rho'' $ also belong to $ K $, $ K $ is the splitting field of $ f $, and the Galois group $ \Gal(K/\Q) \simeq \Z/3 $ is cyclic.

	Another consequence is that $ K $ has units of all signatures. The discriminant of the polynomial $ f $ equals $ \disc_f = (a^2 +3a+9)^2 $. If $ a^2 +3a+9 $ is squarefree (which happens for a~positive density of $ a~$), then $ \cO_K = \Z[\rho] $. The units are small, hence the regulator is also small, and the class number formula implies that the class number is large, roughly $ \approx a^2 $ (up to a~logarithmic factor).

	When we search for indecomposables in a~totally real number field $ K $, it is natural to consider $ K $ in the Minkowski space by the mapping
\[
	\sigma: K \hookrightarrow \R^d, \alpha \mapsto (\sigma_1(\alpha), \sigma_2(\alpha), \dots, \sigma_d(\alpha)).
\]

	For example, consider the situation in a~real quadratic field $ K = \Q(\sqrt{D}) $ with a~fundamental totally positive unit $ \varepsilon $. We can multiply every totally positive element by a~suitable unit to move it into the cone $ \R_{\geq 0}\cdot 1+\R_{> 0}\cdot\varepsilon $ spanned by $ 1 $ and $ \varepsilon $. If $ \beta \succ 1 $ or $ \beta \succ \varepsilon $, then it is not indecomposable, so we can further restrict our attention to the parallelogram
\[
	[0, 1)\cdot 1+[0, 1)\cdot \varepsilon=\{ t_1\cdot 1+t_2\cdot\varepsilon\mid t_1,t_2\in[0,1)\}.
\]

	The situation in totally real cubic fields is similar. The totally positive units form a~discrete set located on the hyperboloid $ \{(x, y, z) \in \R^3 \mid xyz = 1\} $ in the Minkowski space. Up to multiplication by units, each element is contained in the polyhedral cone	
\[
	C = \R_{\geq 0}\cdot1+\R_{\geq 0}\cdot\varepsilon_1+\R_{\geq 0}\cdot\varepsilon_2+\R_{\geq 0}\cdot\varepsilon_1\varepsilon_2,
\]
	where $ \varepsilon_1 $ and $ \varepsilon_2 $ generate the totally positive unit group. This is essentially the content of Shintani's unit theorem~\cite[Thm (9.3)]{Ne}. The cone $ C $ is the union of two ``triangular" cones spanned by 1, $ \varepsilon_1 $, $ \varepsilon_2 $ and $ \varepsilon_1 $, $ \varepsilon_2 $, $ \varepsilon_1\varepsilon_2 $, respectively. Again, we can restrict our search for indecomposables to the parallelepipeds $ [0,1)\cdot 1+[0, 1)\cdot\varepsilon_1+[0, 1)\cdot\varepsilon_2 $ and $ [0, 1)\cdot\varepsilon_1+[0, 1)\cdot\varepsilon_2+[0, 1)\cdot\varepsilon_1\varepsilon_2 $.

	In the simplest cubic fields, this approach (described in more detail in~\cite[Section 4]{KT}) is explicit enough that we can determine all indecomposables.

\begin{theorem}[{\cite[Theorem 1.2]
{KT}}]\label{theoremKT2}
		Let $ K = \Q(\rho) $ be a~simplest cubic field such that $ \cO_K = \Z[\rho] $. Up to multiplication by units, all indecomposables are
		
\begin{itemize}
			\item $ 1 $
			\item $ 1+\rho+\rho^2 $
			\item $ -v-w\rho+(v+1)\rho^2 $, $ 0 \leq v \leq a~$, $ v(a+2)+1 \leq w \leq (v+1)(a+1) $, a~triangle with $ \frac{a^2 +3a+2}{2} $ indecomposables.
\end{itemize}
	
\end{theorem}

	For the indecomposable $ 1+\rho+\rho^2 $, we have	
\[
	\min\left\{\Tr(\delta(1+\rho+\rho^2))\mid \delta\in \cO_K^{\vee +}\right\} = 2.
\]
	For the indecomposables $ \alpha = -v-w\rho+(v+1)\rho^2 $ in the triangle, 
\[
	\min\left\{\Tr(\delta\alpha)\mid \delta \in \cO_K^{\vee +}\right\} = 1.
\]
	
\begin{corollary}[{\cite[Theorem 1.1]
{KT}}]		Let $ K = \Q(\rho) $ be a~simplest cubic field such that $ \cO_K~=~\Z[\rho] $. Then
		
\begin{itemize}
			\item there exists a~diagonal universal form of rank $ \sim 3a^2 $,
			\item any classical universal lattice has rank $ \geq \frac{a^2}{6} $,
			\item any universal lattice has rank $ \geq \frac{a}{3\sqrt{2}} $.
\end{itemize}
	
\end{corollary}

	Gil Muñoz and Tinková~\cite{GMT} recently extended these results to also cover some {non-monogenic} simplest cubic fields.

	\subsection*{Other cubic families}
	Tinková~\cite{Ti1} obtained similar results in other families of cubic fields. The fields in such families are again generated by a~root $ \rho $ of a~cubic polynomial depending on some parameters. More specifically,
	
\begin{itemize}
		\item Ennola's cubic fields, generated by a~root of $ x^3 +(a-1)x^2 -ax-1 $, $ a~\geq 3 $,
		\item a~family investigated by Thomas, generated by a~root of $ x^3 -(a+b)x^2 +abx-1 $.
\end{itemize}
	Tinková considered indecomposables and quadratic forms over the order $ \Z[\rho] $. She showed that, surprisingly, the minimum of $ \Tr(\delta \alpha) $ taken over $ \delta $ in the codifferent can be arbitrarily large for an indecomposable element $ \alpha $. She also applied these results to determine the Pythagoras number (for the simplest cubic fields, it is typically equal to 6)~\cite{Ti2}.

	\subsection*{Continued fraction families of real quadratic fields}
	For comparison, let's discuss similar families in degree two. The two most well-known examples are:
	
\begin{itemize}
		\item Yokoi's family $ \Q(\sqrt{m^2 +4}) $. If $ m = 2n+1 $ is odd, then the relevant continued fraction is
\[
		\frac{1+\sqrt{(2n+1)^2 +4}}{2} = [n+1, \overline{2n+1}].
\]
		\item Chowla's family $ \Q(\sqrt{4m^2 +1}) $.
	
\end{itemize}
	We have also already seen that $ \sqrt{n^2 -1} = [n-1, \overline{1, 2(n-1)}] $.

	The idea is to generalize this by considering families $ \Q(\sqrt{D}) $ where
\[
	\sqrt{D} = [u_0, \overline{u_1, u_2, \dots, u_{s-1}, 2u_0}]
\]
	with $ s $ and $ u_1, \dots, u_{s-1} $ fixed. A necessary condition is that the sequence $ u_1, \dots, u_{s-1} $ must be symmetric, i.e., $ u_i = u_{s-i} $. It turns out that this condition is almost sufficient for the existence of $ D $.
	
\begin{theorem}[{\cite[Theorem]
{Fr}}]		Let $ u_1, \dots, u_{s-1} $ be symmetric, and define the numbers $ q_i $ via:		
\[
		q_{i+1} = u_{i+1}q_i+q_{i-1},\qquad q_{-1}=0, q_0 = 1.
\]
		(This will be the sequence of denominators of the convergents $ \frac{p_i}{q_i} $, and it does \emph{not} depend on $ u_0 $.)

		 There are infinitely many squarefree positive integers $ D \equiv 2, 3 \pmod{4} $ such that \[ \sqrt{D} = [k, \overline{u_1, \dots, u_{s-1}, 2k}] \] if and only if $ q_{s-2} $ or $ \frac{q_{s-2}^2 -(-1)^s}{q_{s-1}} $ is even (otherwise, there is no such $ D $, even when we drop the condition ``squarefree").

		In such a~case, all $ D $ and $ k $ are given by		
\[
		D = D(t) = at^2 +bt+c,\qquad k = k(t) = et+f,\qquad t\geq 1
\]
		for fixed integers $a,b,c,e,f$ that can be explicitly given in terms of $u_i$.
	
\end{theorem}
	There is a~similar characterization for $ D \equiv 1 \pmod 4 $ and the continued fraction expansion of $ \frac{1+\sqrt{D}}{2} $~\cite{H-K}.

	These families have a~number of advantageous properties:
	
\begin{itemize}
		\item The fundamental unit $ \varepsilon $ depends linearly on $ t $.
		\item The class number is large, essentially $ t/\log t $ by the class number formula (see~\cite{CF+,DK,DL}).
		\item Indecomposables behave nicely (as in the simplest cubic fields).
	
\end{itemize}

	\section{Lifting Problem for Universal Forms}\label{sec:8}
	When can a~quadratic form with coefficients in $ \Z $ be universal over a~larger number field $ K $? The answer to this question, which we call the \textit{lifting problem}, seems to be ``very rarely", at least for number fields of small degrees.

	We have already mentioned the result of Siegel on non-universality of sums of squares -- let's now sketch its proof.

\begin{theorem}[{\cite[Theorem I]
{Si2}}]\label{th:8.1}
		If $ K $ is a~totally real number field such that every element of $ \cO_K^+ $ is a~sum of squares, then $ K = \Q $ or $ \Q(\sqrt{5}) $.
	
\end{theorem}

\begin{proof}[Sketch of proof.]
		Assume that every element of $ \cO_K^+ $ is a~sum of squares.

		{Step 1.} We show first that all totally positive units are squares and that we have units of all signatures. By our assumption, each totally positive unit is expressible as $ \varepsilon = \alpha_1^2 +\cdots+\alpha_t^2 $ for some $ \alpha_i \in \cO_K $, and therefore
\[
		1 = N(\varepsilon) \geq \sum_{i=1}^t N(\alpha_i)^2 \geq t,
\]
		which implies $ \varepsilon = \alpha_1^2 $.

		Next, let $ \varepsilon_1, \dots, \varepsilon_{d-1} $ be a~system of fundamental units, and consider the units		
\[
		\varepsilon = (-1)^{a_0}\varepsilon_1^{a_1}\cdots\varepsilon_{d-1}^{a_{d-1}},\qquad a_i \in \{0, 1\}.
\]
		There are $ 2^d $ of them and they have distinct signatures (otherwise there would exist two of these units whose product is totally positive and hence a~square). As there are $ 2^d $ signatures in total, we have units of all signatures.
  
\medskip
		{Step 2.} We show next that every indecomposable is equal to $ \varepsilon^2 $ for some unit $ \varepsilon $. If $ \beta = \sum_{i=1}^t \alpha_i^2 $ is an indecomposable, we must have $ t = 1 $, thus $ \beta = \alpha^2 $ for $ \alpha = \alpha_1 $. Let $ \varepsilon $ be a~unit with the same signature as $ \alpha $, or equivalently, $ \varepsilon\alpha \succ 0 $. If $ \varepsilon\alpha \succ \gamma $ for some totally positive $ \gamma $, then $ \varepsilon^2 \beta \succ \gamma^2 $, and $ \varepsilon^2 \beta $ is not indecomposable. Thus $ \varepsilon\alpha $ is indecomposable. But if $ N(\beta) > 1 $, then $ N(\beta) > N(\varepsilon\alpha) > 1 $ and we found an indecomposable with a~strictly smaller norm greater than 1. We can continue this infinite descent until we reach a~contradiction.
  
\medskip
		{Step 3.} Suppose finally that $ M = \cO_K\setminus \Q(\sqrt{5}) $ is non-empty. Then it contains totally positive elements (because if $ \alpha \in M $, then $ N+\alpha \in M $ is totally positive for a~large enough integer $ N $). Choose a~totally positive element $ \lambda \in M $ with minimal trace. The rest of the proof runs as follows:
		
\begin{itemize}
			\item The fact that $ \lambda $ has minimal trace shows that it is indecomposable. By what we showed earlier, $ \lambda = \varepsilon^2 $ for a~unit $ \varepsilon $. Without loss of generality, we can assume that $ \Tr\varepsilon < 0 $ (substitute $ -\varepsilon $ for $ \varepsilon $ if necessary).
			\item Show $ \Tr\lambda < 3d $, where $ d $ is the degree of $ K $.
			\item Consider the decomposition of $ \varepsilon^2 +\varepsilon+1 \succ 0 $, and after some estimates get a~contradiction.\qedhere
\end{itemize}
\end{proof}

	The general question we are asking is: Over which number fields $ K $ is there a~universal \textit{$ \Z $-form}, i.e., a~positive definite quadratic form with coefficients in $ \Z $? For $ K $ different from $ \Q $ and $ \Q(\sqrt{5}) $, the sum of squares is not universal by Siegel's theorem. In particular, there is no universal diagonal $ \Z $-form (for each diagonal $ \Z $-form is represented by the sum of sufficiently many squares).

	We can extend this to general forms over real quadratic fields, as conjectured by Deutsch~\cite{De1}.

\begin{theorem}[{\cite[Theorem 1.1]
{KY2}}]		\label{theoremKYLifting}
		If $ K \neq \Q(\sqrt{5}) $ is a~real quadratic field, then there is no universal $ \Z $-form over $ K $.
	
\end{theorem}

	For the proof,
	we again want to work with ``\emph{minimal vectors}" of the corresponding quadratic $ \cO $-lattice $ (L, Q) $, as in the proof of Theorem~\ref{theoremKT}. For a~$ \Z $-lattice, they are the vectors $ v $ such that $ Q(v) $ is the smallest represented positive integer. This does not make a~good sense over $ \cO_K $, so we will consider $ \Tr_{K/\Q}(Q(v)) $ (which is a~positive integer).

	Recall that the codifferent is defined as
\[
	\cO_K^\vee = \left\{\delta \in K\mid \Tr(\delta\alpha)\in \Z\,\forall \alpha \in \cO_K\right\}.
\]
	A little more generally, we can look at $ \Tr(\delta Q(v)) $ for $ \delta \in \cO_K^{\vee, +} $, which is still a~non-negative integer. We will be interested in the vectors minimizing this.

	As another tool, let's introduce
	the \emph{tensor product} (see~\cite[Section 4]{KY2} for details for the following discussion) $ (L_1\otimes L_2, Q_1\otimes Q_2) $ of two quadratic $ \Z $-lattices $ (L_1, Q_1) $, $ (L_2, Q_2) $.
	If we fix $ \Z $-bases $v_i$ and $w_j$ for $L_1$ and $L_2$, then $L_1\otimes L_2$ is the $\Z$-lattice whose $ \Z $-basis consists of formal elements $v_i\otimes w_j$.
	The quadratic form is then defined so that
\[
	(Q_1\otimes Q_2)(v\otimes w) = Q_1(v)Q_2(w)\text{ for } v\in L_1, w\in L_2.
\]
	This gives an integral $ \Z $-lattice if at least one of the $ Q_i $ is classical.

\medskip
	An important observation is that if $ Q $ is a~$ \Z $-form of rank $ r $, then $ (\cO_K^r, \Tr(\delta Q)) $ is isomorphic to the tensor product $ (\cO_K\otimes \Z^r, T_\delta \otimes Q) $, where $ T_\delta(x) = \Tr(\delta x^2) $ (to be more precise, we need to identify $ \cO_K $ with $ \Z^d $ by choosing an integral basis).

	A $ \Z $-lattice $ L $ is of \emph{E-type} if for each $ \Z $-lattice $ M $, all the minimal vectors of $ L\otimes M $ are split, i.e., of the form $ v\otimes w $ for $ v \in L $, $ w \in M $. A theorem of Kitaoka~\cite[Theorems 7.1.1, 7.1.2, 7.1.3]{Kit} states that each lattice of rank $ \leq 43 $ is of E-type. The form $ T_\delta $ has rank $ d $, so we are fine when $ d \leq 43 $.

\medskip

	Let $ Q $ be a~universal $ \Z $-form (over a~number field $ K $ of degree $ d \leq 43 $), and let's us introduce the following condition for $ \alpha \in \cO_K^+ $:
\begin{equation}\label{eqConditionA}
		\exists \delta \in \cO_K^{\vee, +}:\;\Tr(\delta\alpha) = \min_{\beta \in \cO_K^+}\Tr(\delta\beta).
\end{equation}
	One uses minimal vectors to show:
	
\begin{itemize}
		\item If $ \alpha $ satisfies~(\ref{eqConditionA}), then $ \alpha $ is a~square (and indecomposable).
		\item Every totally positive unit is a~square, and so there are units of all signatures in $ \cO_K $.
		\item If $ \alpha $ satisfies~(\ref{eqConditionA}), then $ \alpha $ is a~unit.
	
\end{itemize}

\begin{proof}[Sketch of proof of Theorem~$\ref{theoremKYLifting}$]
		Let $ K = \Q(\sqrt{D}) $, so that $ \cO_K = \Z[\omega_D] $, where
\[
		\omega_D =
\begin{cases}
			\sqrt{D},& D \equiv 2,3 \pmod 4,\\
			\frac{1+\sqrt{D}}{2},& D \equiv 1 \pmod 4.
\end{cases}
\]
		We know that $ \cO_K^\vee = \delta\cO_K $ for some totally positive $ \delta $ (because we already established that we have units of all signatures). The elements $ \alpha \in \cO_K^+ $ such that $ \Tr(\delta\alpha) = 1 $ form a~convex set in $ \Z^{d-1} = \Z $ (this set can be described using interlacing polynomials). Hence they form an arithmetic progression of length at least $ \sqrt{D}-1 $. By the remarks preceding the proof, they are all units. But we cannot have more than 4 units in an arithmetic progression in a~real quadratic field by a~result of Newman~\cite{New}. Now only finitely many values of $ D $ need to be examined to finish the proof.
\end{proof}

	Theorem~\ref{theoremKY} is a~partial extension of Theorem~\ref{theoremKYLifting} to higher degrees. How to proceed further? There may be some clever approaches using geometry of numbers, elements of trace 2 in the codifferent, and perhaps even different generators of $ \cO_K^\vee $, but at this point we do not know much more than Theorem~\ref{th:lift} above (although we have a~promising work in progress in this direction).

	As a~few final results, let's mention:
	
\begin{itemize}
		\item Let $K\neq \Q,~\Q(\sqrt 5)$ be a~totally real number field. Then there is no classical universal $ \Z $-form over $K$ of rank 3, 4, or 5~\cite[Corollary 3.4]{KY2}.

		\item Let $F$ be a~totally real number field, $L$ an $\cO_F$-lattice, and $d,m\in\Z_{>0}$.
		There are at most finitely many totally real number fields $K\supset F$ of degree $d=[K:\Q]$ such that $L\otimes \cO_K$
		represents all elements of $m \cO_K^+$~\cite[Theorem 2]{KY3}.
	
\end{itemize}

	\subsection*{Acknowledgments} I am very grateful to Mikul\' a\v s Zindulka for typesetting my handwritten notes that formed the first draft of this paper, and to Martina Vaváčková for LaTeX help.

	Further, I thank Jakub Krásenský (in particular, for his valuable help with the dyadic considerations in Section~\ref{sec:5}), Mentzelos Melistas and Pavlo Yatsyna, as well as the anonymous referee, for several very helpful comments.

	The author was supported by Czech Science Foundation (GA\v CR) grant 21-00420M and Charles University Research
	Centre program UNCE/SCI/022.

\end{document}